\documentclass{amsart}
\usepackage{hyperref}

\newcommand{\R}{\mathbf{R}}

\newcommand{\N}{\mathbf{N}}

\newcommand{\eps}{\epsilon}

\DeclareMathOperator{\E}{\mathbf{E}}
\DeclareMathOperator{\p}{\mathbf{P}}

\theoremstyle{plain}
\newtheorem{theorem}{Theorem}

\newtheorem{lemma}[theorem]{Lemma}

\theoremstyle{definition}

\theoremstyle{remark}

\begin{document}
\title[A Short and Elementary Proof of the CLT]{A Short and Elementary Proof of the Central Limit Theorem by Individual Swapping}
\author{Calvin Wooyoung Chin}
\date{\today}

\begin{abstract}
We present a short proof of the central limit theorem which is elementary
in the sense that no knowledge of characteristic functions, linear operators,
or other advanced results are needed.
Our proof is based on Lindeberg's trick of swapping a term for a normal
random variable in turn.
The modifications needed to prove the stronger Lindeberg--Feller central limit
theorem are addressed at the end.
\end{abstract}

\thanks{This work is supported by the National Research Foundation of Korea
grants NRF-2019R1A5A1028324 and NRF-2017R1A2B2001952.
The author would like to thank the reviewers for valuable suggestions.
He would also like to thank his friend Kihyun Kim for the comments he gave
on the first version of this manuscript.}

\maketitle

\section{Introduction.}

Let $Z$ be a continuous random variable with density $f_Z\colon\R\to\R$
given by
\[ f_Z(t) = \frac{e^{-t^2/2}}{\sqrt{2\pi}}
\qquad \text{for all $t \in \R$.} \]
We call $Z$ a \emph{standard normal} random variable.
The celebrated Lindeberg--L\'evy central limit theorem (CLT) reads as follows.

\begin{theorem}[Central limit theorem] \label{thm:clt}
Let $X_1,X_2,\ldots$ be independent and identically distributed
real-valued random variables having mean $0$ and variance $1$.
Then
\[
\lim_{n\to\infty} \p\biggl(\frac{X_1+\cdots+X_n}{\sqrt{n}} \le x\biggr)
= \p(Z \le x) \qquad \text{for all $x \in \R$.}
\]
\end{theorem}

Theorem \ref{thm:clt} explains why the ``bell-shaped" curve appears
frequently in the histograms of many natural populations.
It also allows us to apply many probabilistic or statistical methods
that work for normal distributions to problems regarding other
types of distributions.

Various proofs of Theorem \ref{thm:clt} often involve characteristic functions,
which are Fourier transforms under a different name.
Due to the relative difficulty of building the theory of characteristic
functions, the proof of Theorem \ref{thm:clt} is often deferred to
a graduate course in probability.

There have been attempts to prove Theorem \ref{thm:clt} without using
characteristic functions.
Trotter \cite{Tro59} revived the idea Lindeberg \cite{Lin22} used to prove
the central limit theorem, and replaced characteristic functions with
linear operators on some function space $C$.
His proof is elegant, but also abstract due to its reliance on the theory of
operators.
Also, the motivation behind the operator $T_X \colon C \to C$ 
defined for each random variable $X$ by
\[ (T_Xf)(y) = \E[f(X+y)] \qquad \text{for each $y \in \R$} \]
is not so intuitive to nonexperts.

The so-called Stein method \cite{Ste86} (see \cite{Bol82} for its application
to the Berry--Esseen theorem for martingales)
can be used to prove Theorem \ref{thm:clt} by focusing on the identity
\[ \E[f'(Z) - Zf(Z)] = 0 \]
instead of characteristic functions, where $f$ is well behaved in some sense.
At the heart of the method is the Stein continuity theorem
\cite[Theorem 2.2.13]{Tao12}, which asserts that under certain conditions,
\[ \lim_{n\to\infty} \E[f'(S_n)-S_nf(S_n)] = 0 \]
for the same type of $f$ as above implies
\[
\lim_{n\to\infty} \p(S_n \le x) = \p(Z \le x) \qquad \text{for all $x \in \R$.}
\]
Proving and using the Stein continuity theorem makes the proof of Theorem
\ref{thm:clt} significantly less elementary.

A recent paper \cite{DG17} provides yet another proof of Theorem \ref{thm:clt}
avoiding characteristic functions, but the proof is rather lengthy and
is based on the notion of Haar expansions.
Also, there are many proofs (\cite{CPP98}, \cite[Theorem 2.2.8]{Tao12},
and \cite{ZH13}) that avoid characteristic functions but require more conditions,
$\E[|X_1|^3] <\infty$ for example, than Theorem \ref{thm:clt}.

In this note, we provide a short proof of Theorem \ref{thm:clt}
which is elementary in the sense that no knowledge of characteristic functions,
linear operators, or advanced results such as Stein's continuity theorem
are needed. We will, however, use one direction of the so-called
portmanteau theorem, but we believe this result to be at the level of
introductory-level analysis (Trotter's proof, Stein's method, and
\cite{DG17} all rely on this result).

Our proof is based on Lindeberg's trick \cite{Lin22} of swapping a term
for a normal random variable in turn.
Although it is in German, a nice survey \cite{EL14} on Lindeberg's method
exists in the literature.
It actually contains some of the calculations we make in this note.

A natural modification of our proof yields
the Lindeberg--Feller central limit theorem (Theorem \ref{thm:clt_lf} below),
which is an important generalization of Theorem \ref{thm:clt}.
The modification we need will be addressed at the end of this note.

\section{Proof of Theorem \ref{thm:clt}.}

The following is a variant of one direction of the \emph{portmanteau} theorem,
which is an indispensable tool when we deal with the type of convergence
(called \emph{convergence in distribution}) that Theorem \ref{thm:clt} asserts.

\begin{lemma} \label{lem:port}
Let $S_1,S_2,\ldots$ be real-valued random variables.
Assume that
\[ \lim_{n\to\infty} \E[f(S_n)] = \E[f(Z)] \]
for all three-times differentiable $f \colon \R \to \R$
such that $f$, $f'$, $f''$, and $f'''$ are bounded.
Then
\[
\lim_{n\to\infty} \p(S_n \le x) = \p(Z \le x) \qquad \text{for all $x \in \R$.}
\]
\end{lemma}

\begin{proof}
Let $x \in \R$ and $\eps > 0$ be given.
Since $Z$ has a continuous probability density function,
there is some $\eta > 0$ such that
\[ \p(Z\le x)-\eps < \p(Z\le x-\eta) \quad \text{and} \quad
\p(Z\le x+\eta) < \p(Z\le x) + \eps. \]
Let $f,F \colon \R \to [0,1]$ be three-times differentiable functions
such that the first, second, and third derivatives are bounded,
$f(t) = 1$ for $t \le x-\eta$, $f(t) = 0$ for $t \ge x$,
$F(t) = 1$ for $t \le x$, and $F(t) = 0$ for $t \ge x+\eta$.
These types of functions are sometimes called
\emph{smooth transition functions.}

Since
\[ \lim_{n\to\infty} \E[f(S_n)] = \E[f(Z)] \ge \p(Z\le x-\eta)
> \p(Z\le x) - \eps \]
and
\[ \lim_{n\to\infty} \E[F(S_n)] = \E[F(Z)] \le \p(Z\le x+\eta)
< \p(Z\le x) + \eps, \]
there is some $N \in \N$ such that
\[ \p(Z\le x) - \eps < \E[f(S_n)] \le \p(S_n \le x) \le \E[F(S_n)]
< \p(Z\le x) + \eps \]
for all $n \ge N$.
As $\eps$ is arbitrary, the proof is finished.
\end{proof}

Now we are ready for the main body of the proof.

\begin{proof}[Proof of Theorem \ref{thm:clt}]
Let $Y_1,Y_2,\ldots$ be independent and identically distributed
standard normal random variables independent from $X_1,X_2,\ldots\,$.
For each $n \in \N$, let
\[
S_{n,i} := \frac{X_1+\cdots+X_{i-1}+Y_{i+1}+\cdots+Y_n}{\sqrt{n}}
\qquad \text{for $i=1,\ldots,n$}
\]
and
\[
Z_{n,i} := \frac{X_1+\cdots+X_{i}+Y_{i+1}+\cdots+Y_n}{\sqrt{n}}
\qquad \text{for $i=0,\ldots,n$.}
\]
A direct computation of probability density functions tells us that
$Z_{n,0}$ is a standard normal random variable.
So, the sequence $Z_{n,n}, Z_{n,n-1},\ldots, Z_{n,0}$ is a process of
turning $(X_1+\cdots+X_n)/\sqrt{n}$ into a standard normal random variable
by swapping one term at a time.
Our strategy is to show that $Z_{n,i}$ and $Z_{n,i-1}$ have similar
distributions, in some sense, for each $i=1,\ldots,n$.
We will use $S_{n,i}$ as a step between $Z_{n,i}$ and $Z_{n,i-1}$.

Let $f \colon \R \to \R$ be any three-times differentiable function
such that $f$, $f'$, $f''$, and $f'''$ are bounded.
Since $Z_{n,0}$ is standard normal, Lemma \ref{lem:port} tells us that
it is enough to show
\begin{equation} \label{eq:clt_2}
\E[f(Z_{n,n})] - \E[f(Z_{n,0})] \to 0 \qquad \text{as $n \to \infty$.}
\end{equation}
For each $n \in \N$ and $i=1,\ldots,n$, Taylor's theorem yields
\begin{equation} \label{eq:clt_1}
f(Z_{n,i}) - f(S_{n,i}) - \frac{f'(S_{n,i})X_i}{\sqrt{n}}
- \frac{f''(S_{n,i})X_i^2}{2n} = \frac{(f''(C_{n,i})-f''(S_{n,i}))X_i^2}{2n}
\end{equation}
where $C_{n,i}$ lies between $S_{n,i}$ and $Z_{n,i}$.
Let $\eps > 0$ be given. Since $f'''$ is bounded, we can take $\delta > 0$
so that
\[ \delta |f'''(t)| \le \eps \qquad \text{for all $t \in \R$.} \]
Then $|x-y| \le \delta$ implies $|f''(x)-f''(y)| \le \eps$
by the mean value theorem.
For a random variable $R$ and an event $A$, let us write
$\E[R;A] := \E[R1_A]$.
If we denote the right side of \eqref{eq:clt_1} by $R_{n,i}$,
we have
\begin{equation} \label{eq:clt_less}
\E[|R_{n,i}|;|X_i| \le \delta\sqrt{n}] \le \frac{\eps}{2n}\E [X_i^2]
= \frac{\eps}{2n}
\end{equation}
since $|X_i|/\sqrt{n} \le \delta$ implies
$|f''(C_{n,i})-f''(S_{n,i})| \le \eps$.
On the other hand, we have
\begin{equation} \label{eq:clt_more}
\E[|R_{n,i}|; |X_i| > \delta\sqrt{n}] \le 
\frac{M}{n}
\E[X_i^2; |X_i| > \delta\sqrt{n}],
\end{equation}
where $M$ is a finite number such that $|f''(x)| \le M$ for all $x \in \R$.
If $n$ is large enough, then the last display is less than $\eps/2n$,
and so $\E[|R_{n,i}|] \le \eps/n$ for all $i=1,\ldots,n$.
Taking the absolute values of the expectations of both sides of \eqref{eq:clt_1}
while noticing that $S_{n,i}$ and $X_i$ are independent, we have
\[
|\E[f(Z_{n,i})] - \E[f(S_{n,i})] - \E[f''(S_{n,i})]/2n| \le
\eps/n \qquad \text{for all $i = 1,\ldots,n$}
\]
for all large $n$.

The same argument applied to $Z_{n,i-1}$ and $S_{n,i}$
instead of $Z_{n,i}$ and $S_{n,i}$ (so that $Y_i$ plays the role of $X_i$) yields
\[
|\E[f(Z_{n,i-1})] - \E[f(S_{n,i})] - \E[f''(S_{n,i})]/2n| \le
\eps/n \qquad \text{for all $i = 1,\ldots,n$}
\]
for all large $n$.
Conflating the last two displays, we have
\[
|\E[f(Z_{n,i})]-\E[f(Z_{n,i-1})]| \le 2\eps/n \qquad
\text{for all $i = 1,\ldots,n$}
\]
for all large $n$. Summing up the last display for all
$i = 1,\ldots,n$, we have
\[
|\E[f(Z_{n,n})] - \E[f(Z_{n,0})]| \le 2\eps \qquad \text{for all large $n$.}
\]
Since $\eps$ is arbitrary, we have proved \eqref{eq:clt_2}.
\end{proof}

\section{The Lindeberg--Feller CLT.}

In this section, we modify our proof of Theorem~\ref{thm:clt}
to prove the following more general theorem which applies to
nonidentically distributed random variables.

\begin{theorem}[Lindeberg--Feller CLT] \label{thm:clt_lf}
For each $n \in \N$, let
\[ X_{n,1},\ldots,X_{n,m_n} \qquad \text{($m_n \in \N$)} \]
be independent real-valued random variables with mean zero.
If
\[
\sum_{i=1}^{m_n} \E[X_{n,i}^2] = 1 \qquad \text{for all $n\in\N$}
\]
and
\[
\lim_{n\to\infty} \sum_{i=1}^{m_n} \E[X_{n,i}^2; |X_{n,i}| > \delta] = 0
\qquad \text{for all $\delta > 0$,}
\]
then we have
\[
\lim_{n\to\infty} \p\biggl(\sum_{i=1}^{m_n} X_{n,i} \le x\biggr)
= \p(Z \le x) \qquad \text{for all $x \in \R$.}
\]
\end{theorem}

\begin{proof}

Let $Y_{n,1},\ldots,Y_{n,m_n}$ be independent
normal random variables that are independent from $X_{n,1},\ldots,X_{n,m_n}$,
and whose means are zero and variances are $\E[X_{n,1}^2],\ldots,\E[X_{n,m_n}^2]$.
By a normal random variable with mean zero and variance $v \ge 0$,
we mean $\sqrt{v}N$ where $N$ is a standard normal random variable.
We will let $X_{n,i}$ and $Y_{n,i}$ take the roles of $X_i/\sqrt{n}$
and $Y_i/\sqrt{n}$, respectively, and proceed as in the proof of
Theorem~\ref{thm:clt}.

Let
\[ S_{n,i} := X_{n,1} + \cdots + X_{n,i-1} + Y_{n,i+1} + \cdots + Y_{n,m_n}
\quad \text{for $i=1,\ldots,m_n$} \]
and
\[ Z_{n,i} := X_{n,1} + \cdots + X_{n,i} + Y_{n,i+1} + \cdots + Y_{n,m_n}
\quad \text{for $i=0,\ldots,m_n$.} \]
Note that $Z_{n,m_n}=\sum_{i=1}^{m_n} X_{n,i}$ and that $Z_{n,0}$
is standard normal.
Given $\eps > 0$ and a three-times differentiable $f\colon\R\to\R$ with
$f$, $f'$, $f''$, and $f'''$ bounded, we will show that
\begin{equation} \label{eq:lf_diff_small}
|\E[f(Z_{n,m_n})] - \E[f(Z_{n,0})]| \le 2\eps
\quad \text{for all sufficiently large $n$.}
\end{equation}
Then the conclusion will follow from Lemma \ref{lem:port}.

As in \eqref{eq:clt_2}, Taylor's theorem implies
\begin{equation} \label{eq:clt_lf_taylor}
f(Z_{n,i})-f(S_{n,i}) - f'(S_{n,i})X_{n,i} - \frac{1}{2}f''(S_{n,i})X_{n,i}^2
= \frac{1}{2}(f''(C_{n,i})-f''(S_{n,i}))X_{n,i}^2
\end{equation}
where $C_{n,i}$ lies between $S_{n,i}$ and $Z_{n,i}$.
Let $R_{n,i}$ denote the right side.
In the same way we had \eqref{eq:clt_less} and \eqref{eq:clt_more} above, we have
\[ \E[|R_{n,i}|;|X_{n,i}|\le \delta] \le \frac{\eps}{2}\E[X_{n,i}^2] \]
and
\[ \E[|R_{n,i}|;|X_{n,i}|> \delta] \le M\E[X_{n,i}^2;|X_{n,i}|>\delta] \]
where $\delta$ and $M$ are taken as in the proof of Theorem~\ref{thm:clt}.

Taking the absolute values of the expectations of both sides of
\eqref{eq:clt_lf_taylor}, we have
\[
\begin{split}
|\E[f(Z_{n,i})] - \E[f(S_{n,i})] - &\E[f''(S_{n,i})]\E[X_{n,i}^2]/2| \\
&\le \E[|R_{n,i}|;|X_{n,i}| \le \delta] + \E[|R_{n,i}|;|X_{n,i}| > \delta] \\
&\le \frac{\eps}{2}\E[X_{n,i}^2] + M \E[X_{n,i}^2;|X_{n,i}| > \delta].
\end{split}
\]
As in the proof of Theorem \ref{thm:clt}, a similar argument yields
\[
\begin{split}
|\E[f(Z_{n,i-1})] - \E[f(S_{n,i})] - &\E[f''(S_{n,i})]\E[Y_{n,i}^2]/2| \\
&\le \frac{\eps}{2}\E[Y_{n,i}^2] + M \E[Y_{n,i}^2;|Y_{n,i}| > \delta].
\end{split}
\]

Since $\E[X_{n,i}^2] = \E[Y_{n,i}^2]$ for all $n$ and $i$,
conflating the last two displays yields
\begin{multline*}
|\E[Z_{n,i}] - \E[Z_{n,i-1}]| \\
\le \eps\E[X_{n,i}^2] + M\E[X_{n,i}^2;|X_{n,i}|>\delta]
+ M\E[Y_{n,i}^2;|Y_{n,i}|>\delta].
\end{multline*}
If we sum up each side for $i=1,\ldots,m_n$, then the first two terms
in the right side tend to $\eps$ and $0$, respectively,
by the assumptions. 
Thus, to establish \eqref{eq:lf_diff_small}, we only need to show
\begin{equation} \label{eq:clt_3}
\lim_{n\to\infty} \sum_{i=1}^{m_n} \E[Y_{n,i}^2;|Y_{n,i}|>\delta] = 0.
\end{equation}

To show \eqref{eq:clt_3}, first observe that for each $\eps > 0$, we have
\[
\begin{split}
\max_{i=1}^{m_n} \E[X_{n,i}^2] &\le \eps^2 + \max_{i=1}^{m_n}
\E[X_{n,i}^2;|X_{n,i}|>\eps] \\
&\le \eps^2 + \sum_{i=1}^{m_n} \E[X_{n,i}^2;|X_{n,i}|>\eps],
\end{split}
\]
whose right side converges to $\eps^2$ as $n \to \infty$.
Thus, we have
\begin{equation} \label{eq:max_null}
\lim_{n\to\infty} \max_{i=1}^{m_n} \E[Y_{n,i}^2]
= \lim_{n\to\infty} \max_{i=1}^{m_n} \E[X_{n,i}^2] = 0.
\end{equation}
Since $\E[Z^2]=1$ and $\E[Z^4]=3$ for standard normal $Z$,
we have $\E[Y_{n,i}^4] = 3(\E[Y_{n,i}^2])^2$
for all $n \in \N$ and $i=1,\ldots,m_n$.
Thus, using the fact that
\[ Y_{n,i}^2 \le \delta^{-2}Y_{n,i}^4 \qquad \text{if $|Y_{n,i}| > \delta$,} \]
we have
\[
\begin{split}
\sum_{i=1}^{m_n} \E[Y_{n,i}^2;|Y_{n,i}|>\delta]
&\le \delta^{-2}\sum_{i=1}^{m_n} \E[Y_{n,i}^4]
= 3\delta^{-2} \sum_{i=1}^{m_n} (\E[Y_{n,i}^2])^2 \\
&\le 3\delta^{-2}\biggl(\max_{i=1}^{m_n} \E[Y_{n,i}^2]\biggr)
\biggl(\sum_{i=1}^{m_n} \E[Y_{n,i}^2]\biggr) \\
&= 3\delta^{-2} \max_{i=1}^{m_n} \E[Y_{n,i}^2].
\end{split}
\]
Since the right side converges to $0$ as $n \to \infty$ by \eqref{eq:max_null},
we obtain \eqref{eq:clt_3}.

\end{proof}

\end{document}